\documentclass[12pt]{article}%
\usepackage{amssymb}
\usepackage{amsmath}%
\usepackage{amsthm}
\usepackage{amsfonts}%
\usepackage{graphicx}
\usepackage{tikz}
\usetikzlibrary{matrix}
\usepackage{enumitem}
\providecommand{\U}[1]{\protect\rule{.1in}{.1in}}
\newtheorem{theorem}{Theorem}[section]

\newtheorem{corollary}[theorem]{Corollary}

\newtheorem{lemma}[theorem]{Lemma}

\newtheorem{proposition}[theorem]{Proposition}

\theoremstyle{definition}
\newtheorem{definition}[theorem]{Definition}
\newtheorem{remark}[theorem]{Remark}

\begin{document}

\title{Tensor splitting properties of $n$-inverse pairs of operators}
\author{Stepan Paul and Caixing Gu}
\maketitle

\begin{abstract}
In this paper we study $n$-inverse pairs of operators on the tensor product of Banach spaces. In particular we show that an $n$-inverse pair of elementary tensors of operators on the tensor product of two Banach spaces can arise \emph{only} from $l$- and $m$-inverse pairs of operators on the individual spaces. This gives a converse to a result of Duggal and M\"{u}ller \cite{DM}, and proves a conjecture of the second named author \cite{GuSM}. Our proof uses techniques from algebraic geometry, which generalize to other relations among operators in a tensor product. We apply this theory to obtain results for $n$-symmetries in a tensor product as well.
%
%
\end{abstract}

\section{Introduction}
Let $B(X)$ be the algebra of all bounded linear operators on a Banach space $X.$ For $S,T\in B(X),$ we define the functional calculus
\begin{equation}\label{betadef}
\beta_{n}(S,T)=\sum\limits_{k=0}^{n}(-1)^{n-k}\binom
{n}{k}S^{k}T^{k}.
\end{equation}
As in Sid Ahmed \cite{Ahmed} and
Duggal and M\"{u}ller \cite{DM}, we say $S$ is a \emph{left $n$-inverse} of $T$ (or $T$ is a
\emph{right $n$-inverse} of $S$, or $(S,T)$ is an \emph{$n$-inverse pair}) if
$\beta_{n}(S,T)
=0$. If $\beta_n(S,T)=0$, but $\beta_{n-1}(S,T)\neq0$, we say $S$ is a \emph{strict} left $n$-inverse of $T$. In fact, these definitions make sense for elements $S$ and $T$ in an arbitrary $\mathbb{C}$-algebra with identity.

This definition is of course a generalization of the definition of an ordinary left inverse---that $S$ is a left inverse of $T$ if and only if
\[
ST-1=0.
\]

Loosely speaking, the expression (\ref{betadef}) is obtained by substituting $S$ for $x$ and $T$ for $y$ in the expansion
\begin{equation*}\label{expansion}
(xy-1)^{n}=\sum_{k=0}^{n}(-1)^{n-k}{\binom{n}{k}}x^{k}y^{k},
\end{equation*}
always keeping powers of $S$ to the left of powers of $T$.




The concept of $n$-inverse pairs of operators 
is motivated by the $n$-isometries studied early in \cite{A2,AHS,AS,R} on
Hilbert spaces and more recently in \cite{BJ,BJ2,COT,Duggal,Gu,GS,ST1} on Hilbert spaces and in \cite{Bay,BMN2,Gu2,HMS} on Banach spaces. An operator
$T$ on a Hilbert space $H$ is called an $n$-isometry if $\beta_{n}(T^{\ast},T)=0$, that is, if $T^{\ast}$ is a left $n$-inverse of $T.$

If $X$ and $Y$ are Banach spaces, we let $X\overline{\otimes}Y$ denote the
completion, endowed with a reasonable uniform cross norm, of the algebraic
tensor product $X\otimes Y$ of $X$ and $Y.$ 
The initial objective of this paper is to prove the following theorem.

\begin{theorem}\label{ninversethm}
Suppose $S_1,T_1\in B(X)$ and $S_2,T_2\in B(Y)$. Then the following are equivalent:
\begin{enumerate}[label=(\alph*)]
\item\label{tensor} $S_1\otimes S_2$ is a (strict) left $n$-inverse of $T_1\otimes T_2$ in $B(X\overline\otimes Y)$.
\item\label{individuals} There exist positive integers $l,m$ with $l+m=n+1$ and $\lambda\in\mathbb{C}^\ast$ so that $S_1$ is a (strict) left $l$-inverse of $\lambda T_1$ in $B(X)$ and $S_2$ is a (strict) left $m$-inverse of $(1/\lambda)T_2$ in $B(Y)$.
\end{enumerate}

\end{theorem}

That \ref{individuals} implies \ref{tensor} was proved by
Duggal and M\"{u}ller in Theorem~2.3 of \cite{DM}. A corollary of their result for $n$-isometric tensor products is proved in
Theorem~2.10 of \cite{Duggal}, which answers questions about $m$-isometric
elementary operators acting on Hilbert-Schmidt operator ideals studied in \cite{BJ} and \cite{BJ2}.

The other implication was conjectured by the second named author in
Conjecture~20 of \cite{GuSM}, and verified for small $n$ and under some technical conditions. 
With some additional work, Theorem 7 of \cite{Gu} for $n$-isometric elementary operators can be viewed as a corollary of this
result.
An elementary operator (acting on Hilbert-Schmidt
operator ideals) of length one is equivalent to the tensor product of two
operators by \cite{Es}. See \cite{Gu} for more general elementary operators
(such as generalized derivations) that are $m$-isometries. 


In this paper, we will prove Theorem~\ref{ninversethm}, and generalize to a more general set of relations among elements of a $\mathbb{C}$-algebra. Specifically, for any polynomial $p(x,y)$, we consider the relation obtained by substituting $S$ for $x$ and $T$ for $y$ into $p(x,y)^n$, always keeping powers of $S$ to the left of powers of $T$. 

Of particular interest are the cases where $p(x,y)=xy-1$ as already discussed, and where $p(x,y)=x-y$. The latter yields
\begin{equation*}\label{gammadef}
\gamma_n(S,T)=\sum_{k=0}^n(-1)^{n-k}{n\choose k}S^kT^{n-k}.
\end{equation*}

This relation is studied in \cite{KK2} and \cite{KK} for bounded operators $S$ and $T$ on a Hilbert space. In this context, we say $T$ is in the $n$th Helton class of $S$ and write $T\in\mathrm{Helton}_n(S)$ if $\gamma_n(S,T)=0$. Furthermore, we say $T$ is an \emph{$n$-symmetry} if $\gamma_n(T^\ast,T)=0$. The $n$-symmetric operators were introduced and studied in connection with
Sturm-Liouville conjugate point theory by Helton \cite{Helton} and studied in \cite{BH}. They
are inspriational in the study of $m$-isometries and more general hereditary
roots in \cite{AHS} and \cite{ST1}. 
Interestingly, we prove in Section~\ref{tensorsplittingsection} that a direct analogue of Theorem~\ref{ninversethm} is possible essentially in exactly the two cases $xy-1$ and $x-y$, and no others.

We consider several applications of this theory in Sections~\ref{opsection} and \ref{anothersection}, and perhaps most interestingly prove the following pair of theorems. If $H$ and $K$ are Hilbert spaces, we denote by $H\overline\otimes K$ the Hilbert space tensor product of $H$ and $K$.

\begin{theorem}\label{nsym}
Suppose $H$ and $K$ are Hilbert spaces, $T_1\in B(H)$ and $T_2\in B(K)$, and both $T_1$ and $T_2$ are left invertible. Then the following are equivalent:
\begin{enumerate}[label=(\alph*)]
\item\label{tensornsym} $T_1\otimes T_2$ is an $n$-symmetry in $B(H\overline\otimes K)$.
\item\label{individualsnsym} There exist positive integers $l,m$ with $l+m=n+1$ and $\lambda\in\mathbb{C}$ with $|\lambda|=1$ so that $\lambda T_1$ is an $l$-symmetry in $B(H)$ and $\bar\lambda T_2$ is an $m$-symmetry in $B(K)$.
\end{enumerate}
\end{theorem}

\begin{theorem}\label{nsym2}
Suppose $H$ and $K$ are Hilbert spaces, $T_1\in B(H)$ and $T_2\in B(K)$. Then the following are equivalent:
\begin{enumerate}[label=(\alph*)]
\item\label{tensornsym} $T_1\otimes I_K+I_H\otimes T_2$ is an $n$-symmetry in $B(H\overline\otimes K)$.
\item\label{individualsnsym} There exist positive integers $l,m$ with $l+m=n+1$ and $\lambda\in\mathbb{C}$ so that $T_1+\lambda I_H$ is an $l$-symmetry in $B(H)$ and $T_2-\lambda I_K$ is an $m$-symmetry in $B(K)$.
\end{enumerate}
\end{theorem}

The outline of this paper is as follows. In Sections~\ref{algsection} and~\ref{tensorsection}, we lay down the
algebraic foundation for dealing with expressions such as (\ref{betadef}) and show that our
problem can be considered in a commutative algebra setting. In Section~\ref{tensorsplittingsection}, we take advantage of the commutativity to prove our main theoretical results. In particular, we
will use Hilbert's Nullstellensatz extensively, and we use the notion of the height of an ideal see that quasi-homogeneous polynomials play a special role. In Section~\ref{opsection}, we need to briefly explain how the more
general algebra results from previous sections imply Theorem \ref{ninversethm}, and show how the theory applies to $n$-symmetries and the Helton class of an operator. Finally in Section~\ref{anothersection}, we study the nilpotent pertubation of a left $n$-inverse. In doing so, we see that the algebraic results apply in a much stronger way for $n$-symmetries, leading to the proof of Theorem~\ref{nsym2}.
%




\section{Definitions and Algebraic Foundation}\label{algsection}

Let $\mathbb{C}$ denote the field
of complex numbers, and let $\mathbb{C}^{\ast}$ denote the set of nonzero
complex numbers. For us, a $\mathbb{C}$-algebra $A$ is a complex vector space which is also an algebra
with an identity. For elements $S,T$ of a $\mathbb{C}$-algebra $A$, we define $\beta_n(S,T)$ and $n$-inverses as in the introduction.

%

Note that if $S$ is a
left $n$-inverse of $T,\ $then $S$ is a left $m$-inverse of $T$ for all $m\geq n.
$ This follows from the recursive formula
\begin{equation}
\beta_{n}(S,T)=S\beta_{n-1}(S,T)T-\beta_{n-1}(S,T).\label{recursive}
\end{equation}
It is also true that $T$ has a left $n$-inverse for all $n$ if and only if $T$ is left-invertible. This follows from (\ref{recursive}) and
\[
\beta_n(S,T)=S\left(\sum_{k=1}^n(-1)^{n-k}\binom{n}{k}S^{k-1}T^k\right)+(-1)^n.
\]
Because of this fact, we avoid the term \emph{$n$-invertible} even though it appears in the literature. We also define $\gamma_n(S,T)$ 
as in the introduction.

We begin by giving a general algebraic formalism for the construction of expressions such as $\beta_n(S,T)$ and $\gamma_n(S,T)$. In particular, we will make precise our loose explanation in the introduction that ``powers of $S$ are kept to the left of powers of $T$''. This is important, for example, because if $A$ is non-commutative, then of course $\beta_n(S,T)$ may not be equal to $(ST-1)^n$. To deal with this discrepancy, we define a vector space homomorphism from the
free commutative $\mathbb{C}$-algebra on $x,y$ to the free $\mathbb{C}%
$-algebra on $X,Y$
\[
\Phi:\mathbb{C}[x,y]\rightarrow\mathbb{C}\langle X,Y\rangle
\]
defined on the monomial basis of $\mathbb{C}[x,y]$ by 
$$\Phi(x^{i}y^{j})=X^{i}Y^{j}.$$
Here, $\mathbb{C}[x,y]$ is the commutative $\mathbb{C}%
$-algebra of formal polynomials in two \emph{commuting} variables $x,y$, and
$\mathbb{C}\langle X,Y\rangle$ is the $\mathbb{C}$-algebra of formal
polynomials in two \emph{non-commuting} variables $X,Y$ (i.e. formal linear
combinations of words in $X,Y$). In what follows, set
\[
R=\mathbb{C}[x,y],\quad F=\mathbb{C}\langle X,Y\rangle.
\]

Again, $\Phi$ is only a vector space homomorphism and not a $\mathbb{C}%
$-algebra homomorphism, so the multiplicative structure is not preserved. For
example, $yx=xy$ in $R$, but
\[
\Phi(yx)=XY\neq YX=\Phi(y)\Phi(x).
\]
However, $\Phi$ is exactly the map we need to construct expressions like $\beta_n$ and $\gamma_n$ because
\[
\Phi\left(  (xy-1)^{n}\right)  
=\beta_n(X,Y),
\quad \Phi\left((x-y)^n)\right)=\gamma_n(X,Y).
\]

Since $F$ is a free object in the category of $\mathbb{C}$-algebras, for any
$\mathbb{C}$-algebra $A$ and $S,T\in A$, there is a unique $\mathbb{C}%
$-algebra homomorphism
\[
\kappa:F\rightarrow A
\]
so that
\[
\kappa(X)=S,\quad\kappa(Y)=T.
\]

For a given element $\omega\in F$, we may write $\omega(S,T)$ for
$\kappa(\omega)$. Using this notation, $S$
is a left $n$-inverse of $T$ if and only if 
$$\Phi((xy-1)^{n})(S,T)=0.$$

Although $\Phi$ is not a $\mathbb{C}$-algebra homomorphism, it does behave like one in a crucial way shown in the following proposition.

\begin{proposition}
\label{prop:phi} The inverse image of any two-sided
ideal in $\mathbb{C}\langle X,Y\rangle$ under $\Phi$ is an ideal in $\mathbb{C}[x,y]$.
\end{proposition}

\begin{proof}
Let $\mathcal I$ be any two-sided ideal in  $\mathbb{C}\langle X,Y\rangle$. Since $\Phi$ is a vector space
homomorphism, it will suffice to check that $\Phi^{-1}(\mathcal I)$ is closed under
multiplication by a monomial.

Suppose $f(x,y)=\sum k_{ij}x^{i}y^{j}$, is such that $f\in\Phi^{-1}(\mathcal I)$, where all but finitely many $k_{ij}$
are nonzero. Then for any
$a,b\geq0$, we have
\[
\Phi(x^{a}y^{b}\cdot f(x,y))=\Phi\left(  \sum k_{ij}x^{i+a}y^{j+b}\right)
=\sum k_{ij}X^{i+a}Y^{j+b}=X^{a}\cdot\left(  \sum k_{ij}X^{i}Y^{j}\right)
\cdot Y^{b}.
\]
The righthand side is clearly in $\mathcal I$, so $x^{a}y^{b}\cdot f(x,y)$ is in
$\Phi^{-1}(\mathcal I)$.
\end{proof}

Notice that the proof of
Proposition~\ref{prop:phi} would not go through with three or more
variables, and indeed the conclusion would not hold. If $\Psi:\mathbb{C}%
[x,y,z]\rightarrow\mathbb{C}\langle X,Y,Z\rangle$ is the analogous vector
space homomorphism for three variables, and $\mathcal I\subset\mathbb{C}\langle
X,Y,Z\rangle$ is the two-sided ideal generated by $XZ$, then $xz\in\Psi
^{-1}(\mathcal I)$ but $xyz\notin\Psi^{-1}(\mathcal I)$.

\begin{lemma}\label{prop:phi2}
Let $I$ be an ideal in $\mathbb{C}[x,y]$. The inverse image under $\Phi$ of the ideal generated by $\Phi(I)$ is equal to $I$.
\end{lemma}

\begin{proof}
This follows immediately from the observation that
\[
\Phi(I)=\Phi(R)\cap\langle\Phi(I)\rangle.
\qedhere
\]
\end{proof}

\section{Tensor Products}\label{tensorsection}

This paper concerns itself with questions about how relations in the tensor product of two possibly non-commutative $\mathbb{C}$-algebras descend to relations in the individual $\mathbb{C}$-algebras and vice versa. In this section, we will build a general framework which will allow  us to use Proposition~\ref{prop:phi} to convert such questions to the commutative algebra setting. Our motivation lies in proving Theorem~\ref{ninversethm}, but in this section we maintain a much more general perspective that can be applied to other relations such as $\gamma_n(S,T)$, and also to nilpotent perturbations in Section~\ref{anothersection}.

In what follows, suppose $A_{1}$ and $A_{2}$ are $\mathbb{C}$-algebras, and
$S_{i},T_{i}\in A_{i}$. Also, define
\[%
\begin{array}
[c]{ccccccccc}%
\kappa_{i}: & F & \rightarrow & A_{i}, & i=1,2
\\
& X & \mapsto & S_{i} 
\\
& Y & \mapsto & T_{i}
\end{array}
\]

Now let
$$\delta:R\rightarrow R\otimes R$$
be any injective $\mathbb{C}$-algebra homomorphism. Letting $S=\mathbb{C}[x_{1},y_{1},x_{2},y_{2}]$, we also
remind ourselves of the $\mathbb{C}$-algebra isomorphism
\begin{equation*}%
\begin{array}
[c]{cccccc}%
\mu & : & R\otimes R & \rightarrow & S & \\
&  & x^{\alpha}y^{\beta}\otimes x^{\gamma}y^{\delta} & \mapsto & x_{1}%
^{\alpha}y_{1}^{\beta}x_{2}^{\gamma}y_{2}^{\delta} &
\end{array}
\label{mueqn}.%
\end{equation*}

We then define $\epsilon$ and $\kappa$ so that the following diagram commutes.

\begin{equation}\label{commdiageq}
\begin{tikzpicture}[baseline=(current  bounding  box.center),->]
 \matrix [matrix of math nodes,row sep=1.5cm,column sep=1.5cm]
 {
 			&	|(R)| R			& |(F)| F			&					\\
	|(S)| S	&	|(RR)| R\otimes R	& |(FF)| F\otimes F	& |(AA)| A_1\otimes A_2	\\
  };
  \tikzstyle{every node}=[midway,auto,font=\small]
  
  \draw (R)		to node {$\Phi$}					(F);
  \draw (R)		to node {$\delta$}					(RR);
  \draw (F)		to node {$\epsilon$}					(FF);
  \draw (RR)	to node {$\Phi\otimes\Phi$}			(FF);
  \draw (F)		to node {$\kappa$}					(AA);
  \draw (FF)	to node {$\kappa_1\otimes\kappa_2$}	(AA);
  \draw (RR) 	to node {$\mu$}					(S);
	
\end{tikzpicture}
\end{equation}

In particular,
\[
\epsilon(X):=(\Phi\otimes\Phi)\delta(x),\quad\epsilon(Y):=(\Phi\otimes\Phi)\delta(y),\quad\kappa:=(\kappa_1\otimes\kappa_2)\epsilon
\]

%
%


Recall the following linear algebra fact about tensor products.

\begin{lemma}
\label{kernelsplit} If $\phi_{1}:V_{1}\rightarrow W_{1}$ and $\phi_{2}%
:V_{2}\rightarrow W_{2}$ are vector space homomorphisms, then
\[
\ker(\phi_{1}\otimes\phi_{2})=(\ker(\phi_{1})\otimes V_{2})+(V_{1}\otimes
\ker(\phi_{2})).
\]

\end{lemma}

\begin{proof}
This is a standard result, but we include the proof for completeness. Let
$V_{i}=V_{i}^{\prime}\oplus\ker(\phi_{i})$ be a splitting of the epimorphism
$\phi_{i}:V_{i}\rightarrow\mathrm{im}(\phi_{i})$ for $i=1,2$. Then
\[
V_{1}\otimes V_{2}=(V_{1}^{\prime}\otimes V_{2}^{\prime})\oplus(V_{1}^{\prime
}\otimes\ker(\phi_{2}))\oplus(\ker(\phi_{1})\otimes V_{2}^{\prime})\oplus
(\ker(\phi_{1})\otimes\ker(\phi_{2})).
\]
However, the restriction $\phi_{1}\otimes\phi_{2}:V_{1}^{\prime}\otimes
V_{2}^{\prime}\rightarrow\mathrm{im}(\phi_{1})\otimes\mathrm{im}(\phi_{2})$ is
an isomorphism because of the splitting. Hence
\begin{align*}
\ker(\phi_{1}\otimes\phi_{2})  & =(V_{1}^{\prime}\otimes\ker(\phi_{2}%
))\oplus(\ker(\phi_{1})\otimes V_{2}^{\prime})\oplus(\ker(\phi_{1})\otimes
\ker(\phi_{2}))\\
& =(\ker(\phi_{1})\otimes V_{2})+(V_{1}\otimes\ker(\phi_{2})).\qedhere
\end{align*}
\end{proof}

For any ideal $I\subseteq R$, we write 
$$I'=\mu(I\otimes R),\qquad I''=\mu(R\otimes I)$$
for the corresponding ideals in $S$. Let 
$$I:=\Phi^{-1}(\ker\kappa_1),\quad J:=\Phi^{-1}(\ker\kappa_2),$$
so that $I$ and $J$ are ideals in $R$ by Lemma~\ref{prop:phi}. We also will define for any $p\in R$,
$$\hat p:=\mu\delta(p).$$

In what follows, we adopt the above notation under the assumptions that $A_1,A_2$ are $\mathbb{C}$ algebras, $S_i,T_i\in A_i$ for $i=1,2$, and that $\delta:R\rightarrow R\otimes R$ is any injective $\mathbb{C}$-algebra homomorphism.

The following theorem allows us to prove facts about relations in $A_1\otimes A_2$ in the commutative algebra $S$.

\begin{theorem}\label{transthm}
Let $p\in R$. Then $\hat p\in I' + J''$ if and only if $\Phi(p)\in\ker\kappa$.
%
\end{theorem}

\begin{proof}
By Lemma~\ref{kernelsplit}, we have
$$\ker(\kappa_1\Phi\otimes\kappa_2\Phi)=I\otimes R + R\otimes J.$$
Hence by the commutativity of (\ref{commdiageq}), $\Phi(p)\in\ker(\kappa)$  if and only of if 
\begin{equation}\label{deltap}
\delta(p)\in I\otimes R + R\otimes J
\end{equation}
Applying the isomorphism $\mu$, (\ref{deltap}) is equivalent to $\hat p\in I'+J''$.
\end{proof}


The following two results will be used in proving all of our major theorems.

\begin{lemma}\label{weaktplem}
Suppose that $p,q_1,q_2\in R$ and
$$\hat p\in\langle q_1(x_1,y_1),q_2(x_2,y_2)\rangle.$$

If 
$\Phi(q_1^l)\in\ker\kappa_1$, $\Phi(q_2^m)\in\ker\kappa_2$, then $\Phi(p^n)\in\ker\kappa$, where $n=l+m-1$.
\end{lemma}

\begin{proof}
By assumption, there exist $f,g\in S$ such that
\begin{equation}\label{generalexpansion}
\hat p^{n}=\sum_{k=0}^{n}\binom{n}{k}f^{n-k}q_1(x_1,y_1)^{n-k}g^kq_2(x_2,y_2)^k.
\end{equation}

Notice that every summand of (\ref{generalexpansion}), is in either $I'$ or $J''$. Thus $\Phi(p^{n})\in\ker\kappa$.
\end{proof}

The following theorem is a variation on the above lemma, which operates under more technical conditions. However, it will allow us to obtain sharp ``if and only if'' statements like Theorems~\ref{ninversethm}, \ref{nsym}, and \ref{nsym2}.

\begin{theorem}\label{iff}
Suppose that $p,q_1,q_2\in R$ are such that there exist $f,g\in R$ so that
\begin{enumerate}
\item $\hat p(x_1,x_2,y_1,y_2)=f(x_2,y_2) q_1(x_1,y_1)+g(x_1,y_1)q_2(x_2,y_2)$.
\item $\Phi(q_1^L)\in\ker\kappa_1$ and $\Phi(q_2^M)\in\ker\kappa_2$ for large enough $L$ and $M$. 
\item\label{notfunky} For any $i,j$, $f^iq_2^j\in J$ if and only if $q_2^j\in J$, and $g^iq_1^j\in I$ if and only if $q_1^j\in I$.
\end{enumerate}
Then $\Phi(p^n)\in\ker\kappa$ if and only if there exist $l$ and $m$ so that $l+m=n+1$ and $\Phi(q_1^l)\in \ker\kappa_1$ and $\Phi(q_2^m)\in \ker\kappa_2$.
\end{theorem}

\begin{proof}
The ``if'' part of the statement follows directly from Lemma~\ref{weaktplem}. In fact, if we assume that $l,m$ are minimal so that $\Phi(q_1^l)\in\ker\kappa_1$ and $\Phi(q_2^m)\in\ker\kappa_2$, then applying (\ref{generalexpansion}) with $n=l+m-1$, we have
$$\hat p^{l+m-2}=f(x_2,y_2)^{l-1}q_1(x_1,y_1)^{l-1}g(x_1,y_1)^{m-1}q_2(x_2,y_2)^{m-1}\text{ in $S/(I'+J'').$}$$
Applying $\mu^{-1}$ to both sides of the expression, we have
\begin{equation*}\label{epsiloneq}
\delta(p^{l+m-2})=g^{m-1}q_1^{l-1}\otimes f^{l-1}q_2^{m-1}\text{ in $R/I\otimes R/J$.}
\end{equation*}
By Condition~\ref{notfunky}, we see that $g^{m-1}q_1^{l-1}\notin I$ and $f^{l-1}q_2^{m-1}\notin J$, so by the commutativity of (\ref{commdiageq}), $\Phi(p^{l+m-2})\notin\ker\kappa$.


Thus if we assume $\Phi(p^n)\in\ker\kappa$, the minimal $l,m$ chosen above must add to no more than $n+1$, proving the ``only if'' part of the theorem.
\end{proof}

\begin{remark}\label{generate}
Condition~\ref{notfunky} of Lemma~\ref{iff} is a rather technical condition, and we take care to include assumptions throughout the exposition that ensure it will be satisfied. However, we note that Condition~\ref{notfunky} is satisfied when $f,q_2$ generate $R$ and $g,q_1$ generate $R$. This follows from Lemma~\ref{minimality} below.
\end{remark}

\begin{lemma}\label{minimality}
Suppose $g\in R$, and $I\subseteq R$ is an ideal so that $g^m\in I$, but $g^{m-1}\notin I$. If $f\in R$ is such that $f,g$ generate $R$, then for any $k\geq0$, $f^kg^{m-1}\notin I$.
\end{lemma}

\begin{proof}
Since $f,g$ generate $R$, there exist $\alpha,\beta$ so that $\alpha f+\beta g=1$. By way of contradiction, let $k$ be minimial so that $f^kg^{m-1}\in I$. Then 
$$\alpha f^kg^{m-1}+\beta f^{k-1} g^m=(\alpha f +\beta g)f^{k-1}g^{m-1}=f^{k-1}g^{m-1}.$$
But since $g^m\in I$, both sides of the equation must by in $I$, violating the minimality of $k$.
\end{proof}

The following lemma, which relies on Hilbert's Nullstellensatz, will assist us in showing that Condition~2 of Theorem~\ref{iff} is satisfied in all of its applications in later sections.

For an ideal $I\subseteq\mathbb{C}[x_{1},\ldots,x_{n}]$, we write $V(I)$ for
the variety along which the ideal vanishes. That is, 
\[
V(I)=\{\mathbf{a}\in\mathbb{C}^{n}|f(\mathbf{a})=0\;\forall f\in I\}.
\]

\begin{lemma}\label{lem:nullstellensatz}
Suppose $\hat p\in\sqrt{I'+J''}$ in $S$. Then for any $(a,b)\in V(I)$, $\hat p(a,b, x,y)\in \sqrt J$ in $R$, and for any $(c,d)\in V(J)$, $\hat p(x,y,c,d)\in \sqrt I$ in $R$.
\end{lemma}

\begin{proof}
Notice that 
$$V(I^{\prime}+J^{\prime\prime})=\{(a,b,c,d)|(a,b)\in
V(I),\;(c,d)\in V(J)\}.$$
  Since $\hat{p}\in\sqrt{I'+J''}$, $\hat p$ vanishes at
every point of $V(I^{\prime}+J^{\prime\prime})$. 
So if
$(a,b)\in V(I)$, then $p(a,b,c,d)=0$ for all $(c,d)\in V(J)$, and thus by Hilbert's
Nullstellensatz, $p(a,b,x,y)\in\sqrt{J}$. Similarly, if $(c,d)\in V(J)$,
$p(x,y,c,d)\in\sqrt{I}$.
\end{proof}

\section{Tensor-splitting properties}\label{tensorsplittingsection}

In this section, we will apply the results of Section~\ref{algsection} to the case where $\delta:R\rightarrow R\otimes R$ is defined by
\begin{equation}\label{deltadef1}
\delta(x)=x\otimes x,\quad \delta(y)=y\otimes y.
\end{equation}
Applying the definitions of Section~\ref{tensorsection}, we see that $\epsilon(X)=X\otimes X$ and $\epsilon(Y)=Y\otimes Y$, and 
$$\kappa(\omega)=\omega(S_1\otimes S_2,T_1\otimes T_2).$$
In particular,
$$\kappa\Phi((xy-1)^n)=\beta_n(S_1\otimes S_2,T_1\otimes T_2)
.$$

We introduce the notation
\[
p_{\lambda,\mu}(x,y):=p(\lambda x,\mu y),\qquad p_{\lambda}(x,y):=p(x,\lambda
y)
\]
for $p(x,y)\in R$, $\lambda,\mu\in\mathbb{C}$. Notice that 
$$p_{a,b}(x,y)=\hat p(x,y,a,b)=\hat p(a,b,x,y)$$

Now proving Theorem~\ref{ninversethm} entails showing that, for $p(x,y)=xy-1$,
\begin{equation}\label{kappaform}
\Phi(p^{n})\in\ker\kappa\Longleftrightarrow\begin{array}{c}\Phi(p_\lambda^{l})\in\ker\kappa_1\text{
and }\Phi(p_{1/\lambda}^{m})\in\ker\kappa_2\\\text{ for some $l,m>0$ with }l+m=n+1,\;\lambda
\in\mathbb{C}^{\ast}\end{array}.
\end{equation}

It is natural to ask if (\ref{kappaform}), or some weaker version, is true for other polynomials $p\in R$. We therefore make the following definitions.

\begin{definition}
\label{ts} Suppose $p\in R$. If for every pair of $\mathbb{C}$-algebras
$A_{1}$, $A_{2}$, and every $S_{i},T_{i}\in A_{i}$ (not both zero), $i=1,2$
\begin{equation}
\Phi(p^{n})(S_{1}\otimes S_{2},T_{1}\otimes T_{2})=0\;\Longrightarrow\;%
\begin{array}
[c]{cc}%
\Phi(p_{a,b}^{l})(S_{1},T_{1})=0,\;\Phi(p_{c,d}^{m})(S_{2},T_{2})=0 & \\
\text{for some $a,b,c,d\in\mathbb{C}$, large enough $l,m$} &
\end{array}
\label{tsp}%
\end{equation}
we say $p$ has a \emph{weak tensor-splitting property}.

If additionally we can replace the phrase \textquotedblleft large enough $l,m
$\textquotedblright\ in (\ref{tsp}) with \textquotedblleft some positive $l,m$ with
$l+m=n+1$\textquotedblright\ whenever $S_{i}$ is left-invertible and $T_{i}$ is right-invertible, then we say $p$ has a \emph{strong tensor-splitting property}.
\end{definition}

\begin{definition}\label{tp}
Suppose $p\in R$. If for every pair of $\mathbb{C}$-algebras $A_{1}$, $A_{2}
$, and every $S_{i},T_{i}\in A_{i}$, $i=1,2$
\begin{equation}\label{tpp}%
\begin{array}
[c]{c}%
\Phi(p^{l})(S_{1},T_{1})=0,\;\Phi(p^{m})(S_{2},T_{2})=0\\
\end{array}
\Longrightarrow\;%
\begin{array}
[c]{c}%
\Phi(p^{n})(S_{1}\otimes S_{2},T_{1}\otimes T_{2})=0\\
\text{for some large enough }n
\end{array}%
\end{equation}
we say $p$ has a \emph{weak tensor-product property}.

If additionally we can replace the phrase \textquotedblleft large enough
$n$\textquotedblright\ in (\ref{tpp}) with \textquotedblleft
$n= l+m-1$\textquotedblright, then we say $p$ has a \emph{strong tensor-product property}.
\end{definition}

In what follows, we use Theorem~\ref{transthm} to characterize polynomials with the above properties by working in the commutative algebra setting. Note that applying the notation of Section~\ref{algsection} to our situation where $\delta$ is defined in (\ref{deltadef1}), 
$$\hat p(x_1,x_2,y_1,y_2)=p(x_1x_2,y_1y_2).$$

%

Our first main result is the following.

\begin{theorem}
\label{thm:every} Every $p(x,y)\in\mathbb{C}[x,y]$ satisfies a weak
tensor-splitting property.
\end{theorem}

\begin{proof}
Suppose $\Phi(p^n)\in\ker(\kappa)$. By Theorem~\ref{transthm}, this means $\hat p^n\in I'+J''$, where $I=\Phi^{-1}(\ker\kappa_1)$ and $J=\Phi^{-1}(\ker\kappa_2)$. Since $S_i,T_i$ are assumed not to both be zero, $I$ and $J$ are proper ideals of $R$, and so $V(I)$ and $V(J)$ are non-empty by the Weak Nullstellensatz. Applying Lemma~\ref{lem:nullstellensatz}, we see that there exist $a,b,c,d$ so that $p_{c,d}\in\sqrt I$ and $p_{a,b}\in\sqrt J$, and thus (again by Theorem~\ref{transthm}) $\Phi(p_{c,d}^l)\in\ker(\kappa_1)$ and $\Phi(p_{a,b}^m)\in\ker(\kappa_2)$ for large enough $l,m$.
\end{proof}

This theorem may seem surprisingly strong, but as we shall see in Corollary~\ref{cor:algebraic}, it turns out the assumption
that $\hat{p}\in \sqrt{I^{\prime}+J^{\prime\prime}}$ is quite restrictive for
most polynomials $p$.

The \emph{height} of a prime ideal $\mathfrak{p}$ is the length $n$ of a
maximal ascending chain of prime ideals $\mathfrak{p}_{0}\subsetneq
\mathfrak{p}_{1} \subsetneq\cdots\subsetneq\mathfrak{p}_{n}=\mathfrak{p}$. The
height of any ideal $I$ is the infimum of the heights of prime ideals
containing $I$. If $I\subseteq\mathbb{C}[x,y]$, then the height
of $I$ is also equal to the complex codimension of $V(I)$ in $\mathbb{C}^2$. The
only height-$0$ ideal is all of $R$, and height-$2$ ideals in $R$ are ones large
enough that their vanishing set is a finite (or empty) set of points.

If $A$ is any $\mathbb{C}$-algebra, we say $S\in A$ is \emph{algebraic} if $S
$ satisfies a polynomial relation; that is there exist scalars $c_{i}$ so that
$c_{n}S^{n}+\cdots+c_{1}S+c_{0}1=0$. In an operator algebra, the algebraic elements are called \emph{algebraic operators}. In this sense, algebraic operators behave like finite matrices.

\begin{proposition}\label{height1}
Let $A$ be any $\mathbb{C}$-algebra, and suppose $\kappa:F\rightarrow A$ is a $\mathbb{C}$-algebra homomorphism as before. Let $S=\kappa(X)$, $T=\kappa(Y)$. If $\Phi^{-1}(\ker\kappa)$ has height $2$,
then $S$ and $T$ are both algebraic.
\end{proposition}

\begin{proof}
Suppose $I=\Phi^{-1}(\ker\kappa)$ has height 2 so that $V(I)$ is a finite collection of points $p_{i}=(a_{i}%
,b_{i})$, $i=1,\ldots,r$. Hence $\sqrt{{I}}$ is the product of the
maximal ideals $\mathfrak{m}_{p_{i}}=\langle x-a_{i},y-b_{i}\rangle$, so
${I}$ contains $\alpha(x,y)=(x-a_{1})^{k}\cdots(x-a_{r})^{k}$ and
$\beta(x,y)=(y-b_{1})^{k}\cdots(y-b_{r})^{k}$ for some $k$. But then
$\Phi(\alpha),\Phi(\beta)\in\ker(\kappa)$, so $\Phi(\alpha)(S)=0$ and
$\Phi(\beta)(T)=0$, and thus $S$ and $T$ are algebraic.
\end{proof}

Because of Proposition~\ref{height1}, the only ``interesting'' (i.e. not both algebraic) pairs $S,T$ in an algebra $A$ are ones for which $I$ has height 1. In what follows, we characterize polynomials $p\in R$ for which a height 1 ideal may arise in our context.

A polynomial $q(x_{1},\ldots,x_{n})\in\mathbb{C}[x_{1},\ldots,x_{n}]$ is
called \emph{quasi-homogeneous} if there exist coprime integers $w_{1}%
,\ldots,w_{n}$ called \emph{weights} and $d$ called the \emph{quasi-degree} so that
for any $\lambda\in\mathbb{C}^{\ast}$,
\[
q(\lambda^{w_{1}}x_{1},\ldots,\lambda^{w_{n}}x_{n})=\lambda^{d}q(x_{1}%
,\ldots,x_{n}).
\]
Equivalently, if we write
\[
q(x_{1},\ldots,x_{n})=\sum_{i_{1},\ldots,i_{n}}k_{i_{1},\ldots,i_{n}}%
x_{1}^{i_{1}}\cdots x_{n}^{i_{n}},
\]
then $k_{i_{1},\ldots,i_{n}}\neq0$ implies $w_{1}i_{1}+\cdots w_{n}i_{n}=d$.
Note that some authors restrict to the case where $w_{1},\ldots,w_{n},d$ are
all positive, but we do not make that part of the definition.

\begin{proposition}\label{tpequiv}
Let $p(x,y)\in R$ be irreducible. Then the following are equivalent:
\begin{enumerate}[label=(\alph*)]
\item\label{qh} $p$ is quasi-homogeneous.
\item\label{forms} $p(x,y)=A(x^\alpha y^\beta-B)$ or $p(x,y)=A(B x^\alpha - y^\beta)$ for some coprime positive integers $\alpha$ and $\beta$ and constants $A,B\in\mathbb{C}$.
\item\label{primes} $\hat p\in\langle p\rangle'+\langle p\rangle''$.
\item\label{ij} There exist height 1 ideals $I,J\subset R$ so that $\hat p\in \sqrt{I' + J''}$.
\item\label{str} $p$ has a strong tensor-product property.
\item\label{weak} $p$ has a weak tensor-product property.
\end{enumerate}
\end{proposition}

\begin{proof}
That \ref{qh} is equivalent to \ref{forms} is a standard result, and we omit the proof.

If $\alpha=\beta=1$ and $A=B=1$, then the implication \ref{forms} implies \ref{primes} follows from
\begin{equation}\label{decomps}
\begin{array}{rcl}
x_{1}x_{2}y_{1}y_{2}-1&=&(x_{1}y_{1}-1)x_{2}y_{2}+(x_{2}y_{2}-1),\\
x_{1}x_{2}-y_{1}y_{2}&=&(x_{1}-y_{1})x_{2}+(x_{2}-y_{2})y_{1}.
\end{array}
\end{equation}
The general case follows from substituting $x^\alpha$ for $x$ and $y^\beta/B$ for $y$.

Condition \ref{primes} clearly implies \ref{ij}, and \ref{str} clearly implies \ref{weak}. By Lemma~\ref{weaktplem}, \ref{primes} implies \ref{str}.

If $p$ has a weak tensor-product property, then let $A_1=A_2=F/\Phi(p)$, with canonical maps $\kappa_i:F\rightarrow A_i$. Then $I=J=\langle p\rangle$ are height 1 ideals, and since $\Phi(p)\in\ker\kappa_i$ for $i=1,2$, the weak tensor-product property implies $\hat p\in\sqrt{ I'+J''}$. Thus \ref{weak} implies \ref{ij}.

Finally, we must prove \ref{ij} implies \ref{qh}. Suppose $I$ and $J$ are height-$1$ ideals in $R$ and $\hat{p}^{n}\in
I^{\prime}+J^{\prime\prime}$. First, note that if $p(x,y)=x$ or $p(x,y)=y$,
then $p$ is quasi-homogeneous. Otherwise, the vanishing set $V(J)$ contains
some $(c,d)$ with $c,d$ both nonzero. Then by Lemma~\ref{lem:nullstellensatz},
we have $p_{c,d}\in\sqrt{I}$. Since $I$ has height $1 $, $\sqrt{I}$ has height 1,
and $\langle p_{c,d}\rangle\subseteq\sqrt{I}$. But since $p_{c,d}$ is
irreducible, $\langle p_{c,d}\rangle$ is prime of height $1$, so $\sqrt
{I}=\langle p_{c,d}\rangle$, and $V(I)$ is the codimension $1$ set of $(a,b)$ so that
$p_{c,d}(a,b)=0$.

Now suppose $(a_{0},b_{0})\in V(I)$ with $a_{0},b_{0}$ nonzero, and let
\[
p(x,y)=\sum_{i,j} k_{ij}x^{i}y^{j}%
\]
(where all but finitely many $k_{ij}$ are zero). Let $(a,b)\in V(I)$ be such
that the constants $\lambda= a/a_{0}$ and $\mu=b/b_{0}$ are nonzero and have
modulus other than $1$.

Then since $p_{a,b}$ and $p_{a_{0},b_{0}}$ are both irreducible and each generate $\sqrt J$, the
polynomials must be constant multiples of one another. Hence, there is a
constant $M$ so that if $k_{i,j}\neq0$, $Ma_{0}^{i}b_{0}^{j}=a^{i}b^{j}$; that
is, $M=\lambda^{i}\mu^{j}$. Therefore, if $k_{i_{1}j_{1}}$ and $k_{i_{2}j_{2}%
}$ are both nonzero, $\lambda^{i_{1}}\mu^{j_{1}}=\lambda^{i_{2}}\mu^{j_{2}}$,
so $\lambda^{i_{1}-i_{2}}=\mu^{j_{2}-j_{1}}$. Let $\alpha,\beta$ be coprime
integers so that $\alpha(i_{1}-i_{2})=\beta(j_{2}-j_{1})$. Then by the Chinese
Remainder Theorem, there exists $\eta$ so that $\eta^{\alpha}=\lambda$,
$\eta^{\beta}=\mu$. Letting $d=\alpha i_{1}+\beta j_{1}$, we have $M=\eta^{d}%
$, and so for any $i,j$ with $k_{i,j}\neq0$, $\eta^{\alpha i+\beta j}=\eta
^{d}$; since $|\eta|\neq1$, $d=\alpha i+\beta j$. Therefore $p_{a_{0},b_{0}}$
and, as a corollary, $p$ are quasi-homogeneous with weights $\alpha$ and
$\beta$ and quasi-degree $d$.
\end{proof}

\begin{corollary}
\label{cor:algebraic}If $p$ is irreducible and not quasi-homogeneous, and
\begin{equation}\label{toostrong}
\Phi(p^{n})(S_{1}\otimes S_{2},T_{1}\otimes T_{2})=0,
\end{equation}
then either $S_{1}$
and $T_{1}$ are both algebraic, or $S_{2}$ and $T_{2}$ are both algebraic.
\end{corollary}


The authors view Corolloary~\ref{cor:algebraic} as an indication that unless $p$ is quasi-homogeneous, then the condition (\ref{toostrong}) is in some sense too strong to be useful. In particular, if $p$ is not quasi-homogeneous, (\ref{toostrong}) implies that one of the two pairs of operators $S_i,T_i$ was uninteresting in the sense that both operators were algebraic. This also sheds new light on Theorem~\ref{thm:every}; the weak tensor-splitting property of $p$ may simply arise as a bi-product of the implied algebraic-ness of the original operators. For this reason, we consider the tensor-splitting property to really only be meaningful in the case where $p$ is quasi-homogeneous.

\begin{theorem}\label{qhimpliessts}
Any irreducible quasi-homogeneous polynomial in two variables has a strong tensor-splitting property.

Specifically, suppose $p(x,y)=A(x^\alpha y^\beta-B)$ or $p(x,y)=A(Bx^\alpha-y^\beta)$ and let $\eta$ be such that $\eta^\beta=B$; then if
$$\Phi(p^n)(S_1\otimes S_2,T_1\otimes T_2)=0$$
and $S_i$ is left-invertible and $T_i$ is right-invertible for $i=1,2$, then there exists $\lambda\in\mathbb{C}^*$ and $l,m\geq0$ with $l+m\leq n+1$ so that
$$\Phi(p_\lambda^l)(S_1,T_1)=0,\qquad \Phi(p^m_{\eta/\lambda})(S_2,T_2)=0.$$
\end{theorem}

\begin{proof}
We will restrict to the cases that $p(x,y)=x-y$ and $p(x,y)=xy-1$. To go to the general cases, substitute $x^\alpha$ and $y^\beta/B$ for $x$ and $y$ respectively.

Suppose $\Phi(p^n)\in\ker(\kappa)$. Since we can assume that $S_i,T_i$ both have one-sided inverses in $A_i$ for $i=1,2$, then $I=\Phi^{-1}(\ker\kappa_1)$ and $J=\Phi^{-1}(\ker\kappa_2)$ are both proper ideals, and there exist points $(a,b)\in V(I)$, $(c,d)\in V(J)$ so that $a,b,c,d$ are all nonzero.


We will now prove that $p,p_\lambda,p_{1/\lambda}$ satisfy the conditions of Theorem~\ref{iff}, whose conclusion is exactly what we are trying to prove. The following variation on (\ref{decomps}) shows Condition 1, that $\hat p\in\langle p_\lambda,p_{1/\lambda}\rangle$, is satisfied.
\begin{equation*}
\begin{array}{rcl}
x_{1}x_{2}y_{1}y_{2}-1&=&(x_{2}y_{2}/\lambda)(\lambda x_{1}y_{1}-1)+1(x_{2}y_{2}/\lambda-1),\\
x_{1}x_{2}-y_{1}y_{2}&=&x_{2}(x_{1}-\lambda y_{1})+\lambda y_1(x_{2}-y_{2}/\lambda).
\end{array}
\end{equation*}
We also just showed Condition 2, that $p_\lambda^L\in I$ and $p_{1/\lambda}^M\in J$, is satisfied as well.

For Condition 3, we treat the two cases of $xy-1$ and $x-y$ separately. If $p(x,y)=xy-1$, then  since $xy/\lambda, xy/\lambda-1$ generate $R$, as do $1,\lambda xy-1$, we may apply Remark~\ref{minimality}.

If $p(x,y)=x-y$, we prove Condition 3 directly. Suppose $(\lambda y)^i(x-\lambda y)^j\in I$. Then \begin{align*}
0&=\kappa_2\Phi((\lambda y)^i(x-\lambda y)^j)\\
&=\lambda^i\kappa_2(\Phi((x-\lambda y)^j)Y^i)\\
&=\lambda^i\kappa_2(\Phi((x-\lambda y)^j))T_2^i
\end{align*}
Since $T_2$ is right-invertible, this implies that $\Phi((x-\lambda y)^j)\in\ker\kappa_2$, so $(x-\lambda y)^j\in I$. By a similar argument, if $x^i(x-y/\lambda)^j\in J$, then $(x-y/\lambda)^j\in J$.
\end{proof}

\begin{remark}\label{unnecessary}
Note that as long as $A_1,A_2$ are nonzero, we do not need to include the invertibility of $S_i,T_i$ as a separate condition for the case of $p(x,y)=xy-1$. Indeed, the only time this condition is invoked in the proof is to show that $a,b,c,d$ are nonzero, but this is implied by the fact that $(abxy-1)^L\in I$ and $(cdxy-1)^M\in J$ and that $I,J$ are proper ideals.
\end{remark}

\section{Operator Theoretic Results}\label{opsection}

%
%

We now briefly explain the operator theoretic consequences of the results of the previous section.

\begin{lemma}\label{complemma}
For Banach algebras $B_1,B_2$, we let
$$\iota:B_1\otimes B_2\rightarrow B_1\overline\otimes B_2$$
be the inclusion of the algebraic tensor product into its completion with respect to some reasonable uniform cross norm. Then for any $\omega\in F$ and $\alpha,\beta\in B_1\otimes B_2$, we have $\omega(\alpha,\beta)=0$ if and only $\iota(\omega(\alpha,\beta))=0$.
\end{lemma}

\begin{proof}
Follows immediately from the injectivity of $\iota$.
\end{proof}

The above lemma may be somewhat trivial, but for us it means that results from earlier sections regarding the algebraic tensor product of $\mathbb{C}$-algebras apply equally well to the (completed) tensor product of Banach algebras.


\begin{proof}[Proof of Theorem \ref{ninversethm}]
Note that for Banach spaces
$X$ and $Y,$ $B(X)$ and $B(Y)$ are nonzero $\mathbb{C}$-algebras with identities. The Banach algebra $B(X\overline\otimes Y)$ is isomorphic to a completion of $B(X)\otimes B(Y)$ with respect to a reasonable uniform cross norm. 

Thus, if we apply Theorem~\ref{qhimpliessts} with $A_1=B(X)$ and $A_2=B(Y)$,
the algebraic assumption that
\[
\Phi((xy-1)^{n})(S_{1}\otimes S_{2},T_{1}\otimes T_{2})=0,
\]
in $B(X)\otimes B(Y)$ is equivalent to the same condition in $B(X\overline\otimes Y)$ by Lemma~\ref{complemma}, and so the same conclusion holds.
%
%
%
%

We should also remark that the statement of Theorem~\ref{ninversethm} also holds when the parenthetical strictness condition is added to both \ref{tensor} and \ref{individuals}. This follows from the original statement of Theorem~\ref{ninversethm} along with the recursive condition in (\ref{recursive}).
\end{proof}

We now move on to some applications of the theory to $p(x,y)=x-y$. Recall from the introduction that Helton classes were defined for operators on a Hilbert space, but the definition extends easily to arbitrary $\mathbb{C}$-algebras. In particular if $S,T$ are elements of an arbitrary $\mathbb{C}$-algebra $A$, we say $T\in\mathrm{Helton}_n(S)$ if
$$\Phi((x-y)^n)(S,T)=0.$$

\begin{theorem}\label{heltonthm}
Assume $S_{1},T_{1}\in B(X)$ and $S_{2},T_{2}\in B(Y)$, and that $S_1,S_2$ are left-invertible, and $T_1,T_2$ are right-invertible.Then the following are equivalent.
\begin{enumerate}[label=(\alph*)]
\item The tensor product $T_{1}\otimes T_{2}$ on $X\overline{\otimes
}Y$ belongs to $\mathrm{Helton}_{n}(S_{1}\otimes S_{2})$.
\item There exist $m$
and $l$ such that $m+l=n+1$ and $\lambda\in\mathbb{C}^*$ so that $\lambda T_{1}\in \mathrm{Helton}_{l}(S_{1})$ and
$(1/\lambda)T_{2}$ $\in \mathrm{Helton}_{m}(S_{2}).$
\end{enumerate}
\end{theorem}

\begin{proof}
Using the same logic as in the proof of Theorem~\ref{ninversethm}, we can apply Theorem~\ref{qhimpliessts} and Lemma~\ref{complemma} with $p(x,y)=x-y$.
\end{proof}

In preparation to prove Theorem~\ref{nsym}, we recall the definition of an $n$-symmetry from the introduction that if $H$ is a Hilbert space, we say $T\in B(H)$ is an $n$-symmetry if
$$\gamma_n(S,T)=\Phi((x-y)^n)(T^*,T)=0.$$

In what follows, we write $\sigma(T)$ and $\sigma_{ap}(T)$ for the spectrum and approximate point spectrum of $T$ respectively.

\begin{lemma}\label{lambda1}
If $T$ is not nilpotent and $\lambda T\in\mathrm{Helton}_n(T^*)$, then $|\lambda|=1$.
\end{lemma}

\begin{proof}
Let $r$ be the spectral radius of $T$, so that there exists $\rho\in\sigma_{ap}(T)$ with $|\rho|=r$. Note that by the non-nilpotency condition, $\rho\neq0$. Then there exists a sequence of unit vectors $h_i\in H$ such that $\Vert(T-\rho I_H)h_i\Vert\rightarrow0$ as $i\rightarrow0$.
\begin{align*}
\left\langle \gamma_{n}(T^{\ast},\lambda T)h_{i},h_{i}\right\rangle  &
=\sum_{k=0}^{n}(-1)^{n-k}{\binom{n}{k}}\left\langle (T^{\ast})^{k}(\lambda
T)^{n-k}h_{i},h_{i}\right\rangle \\
& =\sum_{k=0}^{n}(-1)^{n-k}{\binom{n}{k}}\left\langle (\lambda T%
)^{n-k}h_{i},T^{k}h_{i}\right\rangle 
\end{align*}

Thus, as $i\rightarrow\infty$, we have
\begin{align*}
\left\langle \gamma_{n}(T^{\ast},\lambda T)h_{i},h_{i}\right\rangle& \rightarrow\sum_{k=0}^{n}(-1)^{n-k}{\binom{n}{k}}\left\langle (\lambda
\rho)^{n-k}h_{i},\rho^{k}h_{i}\right\rangle \\
&=  \sum_{k=0}^{n}(-1)^{n-k}{\binom{n}{k}}(\lambda \rho)^{n-k}%
\bar\rho^{k}\left\langle h_{i},h_{i}\right\rangle \\
&= (\lambda \rho-\bar\rho)^{n}.
\end{align*}
However, by assumption, $\gamma_n(T^*,\lambda T)=0$ and $\rho\neq0$, so $|\lambda|=1.$
\end{proof}

\begin{proof}[Proof of Theorem~\ref{nsym}]
We can apply Theorem~\ref{heltonthm} directly to show that \ref{individualsnsym} implies \ref{tensornsym} in Theorem~\ref{nsym}. To prove the converse, assume $T_1\otimes T_2$ is an $n$-symmetry, so that Theorem~\ref{heltonthm} tells us
$$\Phi((x-y)^l)(T_1^*,\lambda T_1)=0,\quad \Phi((x-y)^m)(T_2^*,(1/\lambda)T_2)=0.$$

Thus it only remains to be shown that $|\lambda|=1$, but this is the content of Lemma~\ref{lambda1}.
\end{proof}

\section{Nilpotent perturbation of a left $n$-inverse}\label{anothersection}

Inspired by the construction of $n$-isometries using
the sum of isometries and nilpotent operators in \cite{A1} and \cite{BMN1}, the second
named author proved in Theorem 2 of \cite{GuSM} that if $S$ is a left $m$-inverse of $T$ and $Q$ is a nilpotent operator of order $l$ such that $QS=SQ,$ then $%
S+Q $ is an $n$-inverse of $T$, where $n=l+m-1$. In this section we show that the theorem can be extended in a certain context, and prove the
converse using the
algebraic geometry approach developed in previous sections. The same result actually applies in a much stronger way to Helton classes, where we use it to prove Theorem~\ref{nsym2}.

We achieve these theorems using the same framework built in Section~\ref{tensorsection}, but the map $\delta:R\rightarrow R\otimes R$ is different than in Section~\ref{tensorsplittingsection}.

\begin{proposition}\label{weakperturbation}
Let $A_1,A_2$ be $\mathbb{C}$-algebras and $S,T\in A_1$ and $Q\in A_2$ with $S,T$ not nilpotent, $Q$ nonzero. Also let $p\in\mathbb{C}[x,y]$. If $\Phi(p^l)(S+\lambda1, T)=0$ in $A_1$ and $(Q-\lambda1)^m=0$ in $A_2$, then 
$$\Phi(p^n)(S\otimes 1+ 1\otimes Q, T\otimes 1)=0$$
 where $n=l+m-1$.
\end{proposition}

\begin{proof}
We will apply Lemma~\ref{weaktplem} using $\delta:R\rightarrow R\otimes R$ defined by $\delta(x)=x\otimes 1+1\otimes x$ and $\delta(y)=y\otimes 1$. Say $\kappa_1:F\rightarrow A_1$ is defined by $\kappa_1(X)=S$, $\kappa_1(Y)=T$, and $\kappa_2:F\rightarrow A_2$ is defined by $\kappa_2(X)=Q$ and $\kappa_2(Y)=0$.

First notice that if $p\in R$, then $\mu\delta(p)=p(x_1+x_2,y_1)$. Supposing $p(x,y)=\sum k_{ij}x^iy^j$, we can write
\begin{align*}
 p(x_1+x_2,y_1)&=\sum k_{ij} (x_1+x_2)^i y_1^j\\
 &=\sum k_{ij} (x_1+\lambda)^iy_1^j+(x_2-\lambda)g_{ij}\quad\text{for some $g_{ij}$}\\
 &=\left(\sum k_{ij} (x_1+\lambda)^iy_1^j\right)+ (x_2-\lambda)g\quad\text{for some $g$}\\
 &=p(x_1+\lambda,y_1) + (x_2-\lambda)g
\end{align*}

Let $q_1(x,y)=p(x+\lambda,y)$ and $q_2(x,y)=x-\lambda$, and we can see then that $\hat p\in\langle q_1(x_1,y_1),q_2(x_2,y_2)\rangle$. Also, notice that $\Phi(q_1^l)\in\ker\kappa_1$ and $\Phi(q_2^m)\in\ker\kappa_2$ by assumption. Thus by Lemma~\ref{weaktplem}, $\Phi(p^n)\in\ker\kappa$.
\end{proof}

We say $p(x,y)$ is \emph{linear in $x$} if $p(x,y)$ can be written as $\alpha(y)x+\beta(y)$ for some $\alpha,\beta$.

\begin{theorem}\label{strongperturbation}
Let $A_1,A_2$ be $\mathbb{C}$-algebras and $S,T\in A_1$ and $Q\in A_2$ with $S,T$ not both zero and $Q$ nonzero. Also let $p\in\mathbb{C}[x,y]$ be any irreducible polynomial. Then the following are equivalent:

\begin{enumerate}[label=(\alph*)]
\item\label{another1} $\Phi(p^n)(S\otimes 1+ 1\otimes Q, T\otimes 1)=0$.
\item\label{another2} There exist positive integers $l,m$ with $l+m=n+1$ and $\lambda\in\mathbb{C}$ so that $\Phi(p^l)(S+\lambda1,T)=0$ and $(Q-\lambda1)^m=0$.
\end{enumerate}
\end{theorem}

\begin{proof}
We will use the same definitions of $\delta$ and $\kappa_i$ as in the proof of Proposition~\ref{weakperturbation}. 

Because of Proposition~\ref{weakperturbation}, we need only prove that \ref{another1} implies \ref{another2}. Since $Q$ is assumed to be nonzero, $\ker\kappa_2$ is a proper ideal, so $V(J)$ is non-empty, where $J=\Phi^{-1}(\ker\kappa_2)$. Therefore, let $(\lambda,d)$ be any point in $V(J)$ and set 
$$q_1(x,y):=\hat p(x,y,\lambda,d)=p(x+\lambda,y).$$
Then by Lemma~\ref{lem:nullstellensatz}, $q_1\in\sqrt I$.

By the linearity assumption we can write $p(x,y)=\alpha(y)x+\beta(y)$. We claim that there exists $(a,b)\in V(I)$ so that $\alpha(b)\neq 0$. Indeed, if no such $(a,b)$ exists, then $\alpha(y)\in\sqrt I$, but then since $q_1(x,y)=\alpha(y)(x+\lambda)+\beta(y)\in\sqrt I$, we also have $\beta(y)\in\sqrt I$, and since $p$ was assumed to be irreducible, this means $I=R$, contrary to the assumption that $S$ and $T$ are not both zero.

Therefore, pick $(a,b)\in V(I)$ so that $\alpha(b)\neq 0$. Then by Lemma~\ref{lem:nullstellensatz},
$$q_2(x,y):=\hat p(a,b,x,y)=p(x+a,b)=p(a+\lambda,b)+\alpha(b)(x-\lambda)=\alpha(b)(x-\lambda)\in\sqrt I.$$
And since $\alpha(b)\neq0$, this means $q_2\in\sqrt I$.

We now show that $\hat p, q_1,q_2$ satisfy the conditions of Theorem~\ref{iff}. The proof of Proposition~\ref{weakperturbation} shows that Condition 1 is satisfied, and we just showed that Conditions 2 is satisfied. We use Remark~\ref{minimality} to prove Condition 3. In particular, $q_2$ and $1$ clearly generate $R$, and since $\alpha$ and $\beta$ share no roots, $q_1=\alpha(y)(x+\lambda)+\beta(y)$ and $\alpha(y)$ also generate $R$.
\end{proof}



%
%

%
%
%
%

\begin{corollary}
\label{conjecturesum}
Let $X,Y$ be Banach spaces, and assume $S,T\in B(X)$ and $Q\in B(Y).$ Then the following are equivalent:
\begin{enumerate}[label=(\alph*)]
\item The tensor sum
$S\otimes I_{Y}+I_{X}\otimes Q$ on $X\overline{\otimes}Y$ is
a strict left $n$-inverse of $T\otimes I_{Y}$
.

\item There exist positive $l$ and $m$ such that $l+m=n-1$ and some constant $\lambda\in\mathbb{C}$ so that $S+\lambda I_{X}$ is a strict left $l$-inverse
of $T$ 
 and $Q-\lambda I_{Y}$
is a nilpotent operator of order $l$.
\end{enumerate}
\end{corollary}
\begin{proof}
Follows directly from Proposition~\ref{strongperturbation} using $p(x,y)=xy-1$ 
 and Lemma~\ref{complemma}.
\end{proof}

Corollary~\ref{conjecturesum} was proven in Theorem~22 of \cite{GuSM} using operator theoretic techniques, but here we see that it follows purely from algebraic considerations. We could make a similar statement about Helton classes, but we will see in Proposition~\ref{new} that a much stronger result is possible when $p(x,y)=x-y$.

It is natural to ask how Corollary~\ref{conjecturesum}  can be applied to $n$-isometries. The following result (stated here in the equivalent tensor
product of operators instead of elementary operators) is proved in Theorem~12 of \cite{Gu}.

\begin{theorem}
\label{main2}Suppose $H$ and $K$ are Hilbert spaces. Assume $S\otimes
I_{K}+I_{H}\otimes Q$ is a strict $n$-isometry and 
\begin{equation}\label{speceq}
\sigma (S\otimes
I_{K}+I_{H}\otimes Q)\neq \left\{ \pm e^{\pm i\alpha }e^{i\theta }\text{ for
some }\alpha ,\theta \in \lbrack 0,2\pi )\right\}.
\end{equation}
Then there exist $m$
and $l$ such that $m+2l=n+2$, and $S+\lambda I_H$ (or $\left( Q-\lambda I_K\right) 
$) is a strict $m$-isometry and $Q+\lambda I_K$ (or $S-\lambda I_H)$ is a
nilpotent operator of order $l$ for some constant $\lambda\in\mathbb{C}$.
\end{theorem}

The spectral condition (\ref{speceq}) is necessary by Proposition~14 of \cite{Gu}.

Somewhat surprisingly, Proposition~\ref{strongperturbation} can be strengthened significantly when applied to $p(x,y)=x-y$, and so we treat this case separately.

\begin{proposition}\label{new}
Suppose $A_1,A_2$ are $\mathbb{C}$-algebras and $S_i,T_i\in A_i$ are nonzero for $i=1,2$. Then the following are equivalent:
\begin{enumerate}[label=(\alph*)]
\item\label{anew} $\gamma_n(S_1\otimes 1 + 1\otimes S_2,T_1\otimes 1 + 1\otimes T_2)=0$.
\item\label{bnew} There exist $l,m\geq 0$ and $\lambda\in\mathbb{C}$ so that $l+m=n+1$ and $\gamma_l(S_1+\lambda1,T_1)=0$ and $\gamma_m(S_2-\lambda1,T_2)=0$.
\end{enumerate}
\end{proposition}


\begin{proof}
We define $\delta:R\rightarrow R\otimes R$ by $\delta(x)=x\otimes 1+1\otimes x$ and $\delta(y)=y\otimes 1+1\otimes y$ and let $p(x,y)=x-y$. Thus $\hat p(x_1,y_1,x_2,y_2)=x_1+x_2-y_1-y_2$. Let $\kappa_i:F\rightarrow A_i$ via $\kappa_i(X)=S_i$ and $\kappa_i(Y)=T_i$.

Since $S_i,T_i$ are nonzero, $I$ and $J$ are proper ideals. Thus there exists $(c,d)\in V(J)$, and so by Theorem~\ref{thm:every} if we let $\lambda=c-d$, 
$$q_1(x,y):=p(x+\lambda,y)=x-y+\lambda\in\sqrt I.$$

Then let $(a,b)\in V(I)$, so that $p(a+c,b+d)=a-b+\lambda$, and thus $a-b=-\lambda$. Therefore, by Theorem~\ref{thm:every},
$$q_2(x,y):=p(x-\lambda,y)=x-y-\lambda\in\sqrt J.$$

We have now shown that $p,q_1,q_2$ satisfy the conditions of Theorem~\ref{iff}. In particular, $\hat p(x_1,y_1,x_2,y_2)=1q_1(x_1,y_1)+1q_2(x_2,y_2)$, so Condition 1 is satisfied. We just showed that Condition 2 holds, and Condition 3 is satisfied trivially.
\end{proof}

\begin{remark}\label{rem2}
We remark that condition \ref{bnew} of Proposition~\ref{new} could have been equivalently stated as $\gamma_l(S_1+\alpha1,T_1-\beta1)=0$ and $\gamma_m(S_2-\alpha1,T_2+\beta1)=0$ by making the substitution $\alpha +\beta=\lambda$.
\end{remark}


Surprisingly, for $n$-symmetries, we have the following nice result without
any spectral condition. We first state a lemma which follows from formula
(4) in Lemma~7 of \cite{GS}.

\begin{lemma}\label{imaginary}
If $\gamma_n(T^*+\lambda I, T-\lambda I)=0$, then $\lambda$ is pure imaginary so that $T-\lambda I$ is an $n$-symmetry.
\end{lemma}

\begin{proof}
Let $\alpha\in\sigma_{ap}(T)$, so that there exists a sequence of unit vectors $h_i\in H$ so that $\Vert (T-\alpha I)h_i\Vert\rightarrow 0$ as $i\rightarrow0$.
\begin{align*}
\left\langle \gamma_{n}(T^{\ast}+\lambda I, T-\lambda I)h_{i},h_{i}\right\rangle  &
=\sum_{k=0}^{n}(-1)^{n-k}{\binom{n}{k}}\left\langle(T^*+\lambda)^k (T-\lambda I)^{n-k}h_{i},h_{i}\right\rangle \\
&= \sum_{k=0}^n(-1)^{n-k}{\binom{n}{k}}\langle (T-\lambda I)^{n-k}h_i,(T+\bar\lambda)^kh_i\rangle\\
& =\sum_{k=0}^{n}(-1)^{n-k}{\binom{n}{k}}\left\langle (\lambda T%
)^{n-k}h_{i},T^{k}h_{i}\right\rangle 
\end{align*}

Thus, as $i\rightarrow\infty$, we have
\begin{align*}
\left\langle \gamma_{n}(T^{\ast}+\lambda I, T-\lambda I)h_{i},h_{i}\right\rangle& \rightarrow\sum_{k=0}^{n}(-1)^{n-k}{\binom{n}{k}}\left\langle (\alpha-\lambda)h_i,(\alpha+\bar\lambda)h_i\right\rangle \\
&=  \sum_{k=0}^{n}(-1)^{n-k}{\binom{n}{k}}(\alpha-\lambda)^{n-k}%
(\bar\alpha+\lambda)^{k}\left\langle h_{i},h_{i}\right\rangle \\
&= (\alpha-\bar\alpha-2\lambda)^{n}.
\end{align*}

Hence since $\gamma_{n}(T^{\ast}+\lambda I, T-\lambda I)=0$, we must have $\lambda=(\alpha-\bar\alpha)/2$, and so $\lambda$ is pure imaginary.

Thus, $T^*+\lambda I=(T-\lambda I)^*$, so the original assumption shows that $T^*-\lambda I$ is an $n$-symmetry.
\end{proof}

\begin{proof}[Proof of Theorem~\ref{nsym2}]
We start by proving that (b) implies (a), so we begin by supposing that
$$\gamma_l(T_1^*+\bar\lambda I_H,T_1+\lambda I_H)=0,\quad \gamma_m(T_2^*-\bar\lambda I_K,T_2-\lambda I_K)=0.$$
Then by Proposition~\ref{new}, if $n=l+m-1$, we have
\begin{align*}
\gamma_n((T_1^*+\bar\lambda I_H)\otimes I_K+I_H\otimes(T_2^*-\bar\lambda I_K),& (T_1+\lambda I_H)\otimes I_K+I_H\otimes (T_2-\lambda I_K))\\
&=\gamma_n(T_1^*\otimes I_K+I_H\otimes T_2^*, T_1\otimes I_K+I_H\otimes T_2)\\
&=0
\end{align*}
which proves the claim.

To prove that (a) implies (b), we assume 
$$\gamma_n(T_1^*\otimes I_K+I_H\otimes T_2^*, T_1\otimes I_K+I_H\otimes T_2)=0$$
Using Remark~\ref{rem2} and Proposition~\ref{new} we can conclude that for some $\lambda$ and some positive $l,m$ with $l+m=n+1$, we have
$$\gamma_l(T_1^*-\lambda I_H,T_1+\lambda I_H)=0,\quad \gamma_m(T_2^*+\lambda I_K,T_2-\lambda I_H)=0$$
Finally, applying Lemma~\ref{imaginary}, we see that $\lambda$ is pure imaginary so that $\bar\lambda=-\lambda$, and the claim is proved.
\end{proof}

Stepan Paul

Department of Mathematics

University of California, Santa Barbara

Santa Barbara, CA 93106

e-mail: spaul@math.ucsb.edu

\bigskip

Caixing Gu

Department of Mathematics

California Polytechnic State University

San Luis Obispo, CA 93407

e-mail: cgu@calpoly.edu

\bigskip

Mathematics Subject Classification (2010). Primary 47A80, 47A10, Secondary
47B47\bigskip

Keywords: Banach algebra, algebraic operator, quasi-homogeneous polynomial, $n$-inverse, $n$-isometry, tensor product 

\begin{thebibliography}{99}                                                                                               %
\bibitem{A1} J. Agler, Sub-Jordan operators: Bishop's theorem, spectral
inclusion, and spectral sets, J. Operator Theory 7 (2)(1982) 373-395.

\bibitem {A2}J. Agler, \textit{A disconjugacy theorem for Toeplitz operators,
}American Journal of Mathematics, 112 (1990) 1--14.

\bibitem {AHS}J. Agler, W. Helton, M. Stankus, \textit{Classification of
hereditary matrices,} Linear Algebra Appl. 274 (1998) 125--160.

\bibitem {AS}J. Agler, M. Stankus, \textit{$m$-isometric transformations of
Hilbert space, I}, Integral Equations Operator Theory 21 (4)(1995) 383--429.

\bibitem {BH}J. A. Ball and J. W. Helton, Nonnormal dilations, disconjugacy, and constrained
spectral factorization, Integral Equations Operator Theory 312 (1980) 216--309.

\bibitem {Bay}F. Bayart, \textit{$m$-Isometries on Banach spaces}, Math.
Nachr. 284 (2011), 2141--2147.

\bibitem{BMN1} T. Berm\'{u}dez, A. Martin\'{o}n and J. Noda,\ An isometry
plus a nilpotent operator is an $m$-isometry and applications, J. Math.
Anal. Appl. 407(2) (2013) 505-512.

\bibitem {BMN2}T. Berm\'{u}dez, A. Martin\'{o}n, J. Noda, \textit{Products of
$m$-isometries,} Linear Algebra Appl. 438 (2013) 80--86.

\bibitem {BJ}F. Botelho, J. Jamison, \textit{Isometric properties of
elementary operators,} Linear Algebra Appl. 432 (2010) 357--365.

\bibitem {BJ2}F. Botelho, J. Jamison and B. Zheng, \textit{Strict isometries
of any orders}, Linear Algebra Appl. 436 (2012) 3303--3314.


\bibitem {COT}M. Ch\={o}, S. \^{O}ta and K. Tanahashi, \textit{Invertible
weighted shift operators which are $m$-isometries}, Proc. Amer. Math.
Soc. 141 (2013) 4241--4247.

\bibitem {Duggal}B. P. Duggal, \textit{Tensor product of $n$-isometries,} Linear Algebra Appl. 437 (1) (2012) 307--318.

\bibitem {DM}B. P. Duggal and V. M\"{u}ller, \textit{Tensor product of left
$n$-invertible operators}, Studia Math. 215 (2013), 113--125.

\bibitem {Gu}C. Gu, \textit{Elementary operators which are $m$-isometries,} Linear Algebra Appl. 451 (2014) 49--64.

\bibitem {Gu2}C. Gu, \textit{The $(m,q)$-isometric weighted shifts on $l_{p}$-spaces}, Integral Equations Operator Theory 82 (2015) 157--187.

\bibitem {GuSM}C. Gu, \textit{Structures of left $n$-invertible operators and their applications}, Studia Math. 226 (3) (2015) 189--211.

\bibitem {GS}C. Gu and M. Stankus, \textit{Some results on higher order
isometries and symmetries: products and sums with a nilpotent}, Linear Algebra
Appl. 469 (2015) 500--509.

\bibitem {Es}J. Eschmeier, \textit{Tensor products and elementary operators},
J. Reine Angew. Math. 390 (1988) 47--66.

\bibitem {Helton}J.W. Helton, \textit{Jordan operators in infinite dimensions
and Sturm-Liouville conjugate point theory,} Trans. Amer. Math. Soc. 170
(1972) 305--331.

\bibitem {HMS}P. Hoffman, M. Mackey and M. \'{O} Searcoid, \textit{On the
second parameter of an $(m,p)$-isometry}, Integral Equations Operator
Theory 71 (2011) 389--405.

\bibitem {KK2}I. Kim, Y. Kim, E. Ko, J. E. Lee, \textit{Some connections
between an operator and its Helton class,} J. Math. Anal. Appl. 340 (2) (2008) 1235--1240.

\bibitem {KK}I. Kim, Y. Kim, E. Ko, J. E. Lee, \textit{Inherited properties
through the Helton class of an operator}, Bull. Korean Math. Soc. 48 (1)
(2011) 183--195.


\bibitem {R}S. Richter, \textit{A representation theorem for cyclic analytic
two-isometries,} Trans. Amer. Math. Soc. 328 (1991) 325--349.


\bibitem {Ahmed}O. A. M. Sid Ahmed, \textit{Some properties of $m$-isometries and $m$-invertible operators on Banach
spaces}, Acta Math. Sci. Ser. B Engl. Ed. 32 (2012), 520--530.

\bibitem {ST1}M. Stankus, \textit{$m$-Isometries, $n$-symmetries and
other linear transformations which are hereditary roots,} Integral Equations
Operator Theory 75 (2013) 301--321.


\end{thebibliography}
\end{document}